\documentclass[10pt, reqno]{amsart}
\usepackage{graphicx} 

\usepackage{amsmath}
\usepackage{amssymb}
\usepackage{tikz-cd}
\usepackage{amsthm}
\usepackage{mathrsfs}
\usepackage{fullpage}
\usepackage{float}
\usepackage{keytheorems}
\usepackage[capitalise]{cleveref}
\usepackage{mathabx}
\usepackage{enumitem}
\usepackage{tikz}

\newcommand{\R}{\mathbb{R}}

\newcommand{\Z}{\mathbb{Z}}
\newcommand{\N}{\mathbb{N}}
\newcommand{\Q}{\mathbb{Q}}
\newcommand{\h}{\mathbb{H}}

\DeclareMathOperator{\Stab}{Stab}
\DeclareMathOperator{\Aff}{Aff}

\DeclareMathOperator{\Aut}{Aut}
\DeclareMathOperator{\Axis}{Axis}
\DeclareMathOperator{\Fix}{Fix}

\DeclareMathOperator{\Isom}{Isom}
\DeclareMathOperator{\BS}{BS}
\DeclareMathOperator{\id}{id}

\DeclareMathOperator{\tr}{tr}

\DeclareMathOperator{\Sol}{Sol}

\DeclareMathOperator{\CAT}{CAT}
\DeclareMathOperator{\DL}{DL}
\DeclareMathOperator{\Cay}{Cay}
\DeclareMathOperator{\Hom}{Hom}

\DeclareMathOperator{\Diag}{Diag}

\DeclareMathOperator{\GL}{GL}

\DeclareMathOperator{\Lie}{Lie}
\DeclareMathOperator{\AutIsom}{AutIsom}

\newcommand{\inv}{^{-1}}
\newcommand{\ab}{^{\text{ab}}}

\usepackage{xcolor}
\usepackage[colorinlistoftodos]{todonotes}

\definecolor{todoblue}{RGB}{255, 156, 0}

\usepackage[top=1.2in,bottom=1.2in,left=1.2in,right=1.2in]{geometry}

\newtheorem{theorem}{Theorem}[section]
\newtheorem{proposition}[theorem]{Proposition}
\newtheorem{lemma}[theorem]{Lemma}
\theoremstyle{definition}
\newtheorem{definition}[theorem]{Definition}

\title{Horocyclic products have Y-posets of hyperbolic structures}
\author{Noah Caplinger}
\date{\today}

\usepackage{subfiles} 

\begin{document}

\maketitle

\vspace{-1cm}

\begin{abstract}
    A \textit{hyperbolic structure} on a group $G$ is a (not necessarily properly discontinuous) cobounded action of $G$ on a Gromov hyperbolic space, considered up to coarsely $G$-equivariant quasi-isometry. We show that for groups $G$ acting geometrically and positively on a horocyclic product $X\bowtie Y$, all hyperbolic structures on $G$ come from the two actions on the factors. The ingredients of the proof include Malcev rigidity and a new (``vertical") boundary of horocyclic products. We also give the first example of a group acting geometrically on a horocyclic product of two mixed millefeuille spaces.
\end{abstract}

\vspace{-.2cm}

\section{Introduction}

A \textit{hyperbolic structure} on a group $G$ is a cobounded (not necessarily properly discontinuous) action of $G$ on a Gromov-hyperbolic space, considered up to coarsely $G$-equivariant quasi-isometry. The set\footnote{For set-theoretic reasons, we actually consider certain equivalence classes of generating sets whose Cayley graphs are hyperbolic. See \cref{def:cobounded_poset} and \cref{def:hyp_poset}.} of hyperbolic structures $\mathcal{H}(G)$ admits a partial ordering according to the amount of information the action remembers about the group: $X \preceq Y$ if there is a coarsely equivariant coarsely Lipschitz map $Y \to X$. Hyperbolic structures on $G$ are called \textit{elliptic}, \textit{lineal}, \textit{quasi-parabolic} or \textit{of general type} according to their place in the Gromov classification of hyperbolic actions (see \cref{thm:Gromov_Classification}). Lineal hyperbolic structures are exactly those which are quasi-isometric to $\R$, quasi-parabolic structures are those which fix a unique point at infinity, and elliptic structures are trivial.

The poset of hyperbolic structures has been computed for many particular groups, including the solvable Baumslag-Solitar groups \cite{BS_Actions} $G = \BS(1,n)$, mapping tori of Anosov maps on the torus \cite{AR} $G = \Z^2 \rtimes_A \Z$ (for $|\tr(A)| > 2$), the lamplighter groups \cite{Wreath_Actions} $G = (\Z/n \wr  \Z)$, and certain abelian-by-cyclic groups \cite{Valuations}. In each case, $\mathcal{H}(G)$ has a common form: there is a unique elliptic structure which is dominated by a unique lineal structure $G \curvearrowright L$. Every other hyperbolic structure is quasi-parabolic, dominates the unique lineal structure and lies in one of two subposets $\mathcal{P}_+(G), \mathcal{P}_-(G) \subset \mathcal{H}(G)$ according to whether the unique fixed point is attracting or repelling for some fixed hyperbolic element $t \in G$. These subposets intersect at the unique lineal structure, and each have a maximum element. We call a poset of this form a \textit{$Y$-poset}, see \cref{figure:Y-poset}.

The groups mentioned in the above paragraph have another similarity: they all act geometrically (properly discontinuously and cocompactly) on horocyclic products. Given two Gromov hyperbolic spaces $X$ and $Y$ equipped with Busemann functions $\beta_X$ and $\beta_Y$, their \textit{horocyclic product} is $$X\bowtie Y := \{(x,y) \mid \beta_X(x) = -\beta_Y(y)\}.$$ Actions on horocyclic products were studied extensively by Levitin and the author \cite{CaplingerLevitin} and by Ferragut \cite{Ferragut:Rigidity}, \cite{Ferragut:Visual_Boundary}. In this paper we show that all groups acting geometrically and \textit{positively} (without swapping factors, see \cref{subsec:ThmB}) on horocyclic products have $Y$-posets of hyperbolic structures. We refer to \cref{sec:Prelim} for all undefined terms.

\begin{theorem}
\label{thm:main}

    Let $X$ and $Y$ be proper, geodesically complete $\CAT(-\kappa)$ spaces (for $\kappa > 0$) each equipped with a Busemann function about a non-isolated point at infinity. Give $X\bowtie Y$ a metric arising from an admissible monotone norm. 
    Let $G$ be a finitely-generated group acting geometrically and positively on $X\bowtie Y$. Then

    \begin{enumerate}
        \item There is a unique lineal hyperbolic structure $G\curvearrowright L$ on $G$. 
    \end{enumerate}


    Let $t \in G$ be an element which acts hyperbolically on $L$. Then $t$ also acts hyperbolically on any quasi-parabolic structure. Let $\mathcal{P}_+(G)$ denote the hyperbolic structures that fix the attracting fixed point of $t$, and let $\mathcal{P}_-(G)$ denote the structures that fix the repelling point of $t$.

    \begin{enumerate}[resume]
        \item Let $\ast$ denote the trivial hyperbolic structure on $G$. Then $\mathcal{H}(G) = \{\ast, L\} \cup \mathcal{P}_+(G) \cup \mathcal{P}_-(G)$. The elements of $\mathcal{P}_+(G)\setminus \{L\}$ are incomparable to the elements of $\mathcal{P}_-(G)\setminus \{L\}$.\
        \item The subposet $\mathcal{P}_+(G)$ and $\mathcal{P}_-(G)$ each have maximum elements. Further assume that the element $t$ acts on $X\bowtie Y$ with positive height change\footnote{If $t$ has negative height change, the roles of $X$ and $Y$ may be swapped.}. Then the maximum element of $\mathcal{P}_+(G)$ can be represented by a $G$-action on a space quasi-isometric to $X$. The maximum element of $\mathcal{P}_-(G)$ can be represented by an action on a space quasi-isometric to $Y$. 
    \end{enumerate}
\end{theorem}

\begin{figure}
\begin{center}

\begin{tikzpicture}[
    node font=\large,
    every node/.style={inner sep=0pt},scale=0.85
]
 
\node[draw, circle, minimum size=24pt, line width=0.8pt] (star) at (0,-2) {$*$};
\node[draw, circle, minimum size=24pt, line width=0.8pt] (R) at (0,0) {$\mathbb{R}$};
\node[draw, circle, minimum size=24pt, line width=0.8pt] (X) at (-2.5,3.2) {$X$};
\node[draw, circle, minimum size=24pt, line width=0.8pt] (Y) at ( 2.5,3.2) {$Y$};

\draw[thick] (star) -- (R);
 
\foreach \ang in {115,125,135,145} {
    \draw[thick] (R) -- ++(\ang:1.5);
}
 
\foreach \ang in {-35,-47,-59,-71} {
    \draw[thick] (X) -- ++(\ang:1.5);
}
 
\foreach \ang in {35,45,55,65} {
    \draw[thick] (R) -- ++(\ang:1.5);
}
 
\foreach \ang in {-110,-122,-134,-146} {
    \draw[thick] (Y) -- ++(\ang:1.5);
}
 
\begin{scope}[shift={(-1.25,1.6)}, rotate=38]
    \draw[blue, dotted, line width=1.3pt] (0,0) ellipse (1.05cm and 2.8cm);
\end{scope}
 
\begin{scope}[shift={(1.25,1.6)}, rotate=-38]
    \draw[red, dotted, line width=1.3pt] (0,0) ellipse (1.05cm and 2.8cm);
\end{scope}
 
\node[blue] at (-3.7,0.8) {$\mathcal{P}_+(G)$};
\node[red]  at ( 3.7,0.8) {$\mathcal{P}_-(G)$};
 
\end{tikzpicture}

\end{center}

\caption{\label{figure:Y-poset} A cartoon of the Hasse diagram of a Y-poset}

\end{figure}
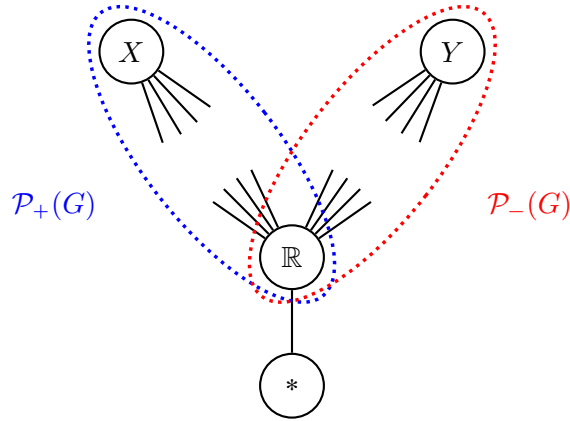

\cref{thm:main} heuristically says that all hyperbolic structures on groups acting geometrically on $X\bowtie Y$ come from the actions on $X$ and $Y$ given by \cite[Theorem B]{CaplingerLevitin}. It does not recover the aforementioned results in that it does not say anything specific about $\mathcal{P}_{\pm}(G)$, while \cite{Valuations}, \cite{BS_Actions}, \cite{Wreath_Actions}, and \cite{AR} all give complete computations of $\mathcal{H}(G)$. \cref{thm:main} does apply to a much wider class of groups than those mentioned above, including 

\begin{enumerate}
    \item mapping tori of expanding maps on nilmanifolds,
    \item mapping tori of many Anosov maps on nilmanifolds (See \cite[Subsection 1.2]{CaplingerLevitin}), 
    \item the cross-wired lamplighter groups of \cite{CWLL} and 
    \item the new group $G_p$ which we introduce in \cref{sec:New_Group}.
\end{enumerate}

Parts 1 and 2 of \cref{thm:main} are similar to Section 4 of Abbott-Balasubramanya-Osin \cite{Valuations} and indeed part of the proof was inspired by \cite[Proposition 4.1]{Valuations}. Unlike in \cite[Proposition 4.1]{Valuations}, we do not know from the outset that our group has virtually cyclic abelianization. This is a relatively easy step in \cite{Valuations}, but is a major part of the proof in our more general setting.

\vspace{.3cm}

\noindent\textbf{Proof outline.} The primary difficulty in proving Part 1 of \cref{thm:main} is in showing that $G\ab$ is virtually $\Z$, the unique $\Z$ factor being represented by the height function $h:G \to \Z$. If there is another $\Z$ factor, there is a homomorphism $k:G\to \Z$ which does not vanish on $\ker(h)$. By \cite[Theorem B]{CaplingerLevitin}, $\ker(h)$ is (up to finite noise) comprised of many finitely-generated torsion-free nilpotent groups on which elements $\{h \neq 0\}$ act by conjugation via Anosov maps. We then use Malcev rigidity to show that no such homomorphism $k$ can exist.

Our proof of Part 3 of \cref{thm:main} involves the use of a new boundary object of a horocyclic product, which we call the \textit{vertical boundary} $\mathcal{V}(X\bowtie Y)$. The vertical boundary is similar to, but distinct from, the visual boundary studied by Ferragut \cite{Ferragut:Visual_Boundary}. If $Z \in \mathcal{P}_+(G)$ dominates the action on $X$, then we show that the induced map $\pi_Z:X\bowtie Y \to Z$ gives a continuous, equivariant map $\partial \pi_Z:\mathcal{V}(X\bowtie Y) \to \partial_\infty Z$. We then use the dynamics of the actions of $G$ on the boundaries to show that $\partial \pi_Z$ must factor through the natural map $\mathcal{V}(X\bowtie Y) \to \partial_\infty X$. This fact is then used to construct a quasi-inverse $X \to Z$, showing that $X$ and $Z$ are equivalent.

\vspace{.3cm}

\noindent\textbf{Organization of the paper.} The paper has five sections. Section 2 is preliminary, introducing hyperbolic structures and horocyclic products. Parts 1 and 2 of \cref{thm:main} are proven in Section 3, and Part 3 is proven in Section 4. In Section 5, we give an example of a group acting geometrically on a horocyclic product of mixed millefeuille spaces.

\vspace{.3cm}

\noindent\textbf{Acknowledgments.} The author would like to thank Benson Farb for his constant encouragement and Denis Osin, Carolyn Abbott, Tullia Dymarz, Daniel Levitin, Daniel Groves, Kevin Whyte and Kevin Wortman for helpful conversations. He would also like to thank Denis Osin for introducing him to \cite{ABO} and thereby the theory of hyperbolic structures, and Daniel Levitin and Benson Farb for their extensive comments on early drafts.

\section{Hyperbolic structures and horocyclic products}
\label{sec:Prelim}

This section introduces the poset of hyperbolic structures on a group and horocyclic products. The first two subsections are based on \cite[Section 3]{ABO}.

\subsection{Generalities on group actions}

Throughout this paper, all actions on metric spaces are assumed to be isometric.

\begin{definition}
    An action of a group $G$ on a metric space $X$ is 
    \begin{enumerate}
        \item \textit{proper} if for every bounded subset $B\subset X$, the \textit{coarse stabilizer} $\{g \in G \mid g\cdot B \cap B \neq \emptyset\}$ is finite.
        \item \textit{cobounded} if there is a subset $B \subset X$ of finite diameter whose $G$-translates cover $X$, meaning $X = \bigcup_{g\in G} g\cdot B$.
        \item \textit{geometric} if it is proper and cobounded.
    \end{enumerate}
\end{definition}

\begin{definition}
    Let $G$ act on metric spaces $X$ and $Y$. A function $f:X\to Y$ is said to be
    \begin{enumerate}
        \item \textit{coarsely $G$-equivariant} if there is some $C > 0$ so that $d_Y(f(g\cdot x), g\cdot f(x)) < C$ for all $g \in G$ and\footnote{\cite{ABO} allows this constant to depend on $x$, but these definitions are equivalent in the case of cobounded actions and this definition is more convenient.} $x \in X$.
        \item \textit{coarsely Lipschitz} if there are constants $A>1, \ B > 0$ so that $$d_Y(f(x),f(y)) \leq Ad(x,y) + B.$$
        \item a \textit{dominating map} if $f$ is both coarsely $G$-equivariant and coarsely Lipschitz.\footnote{\cite{ABO} uses a slightly different definition, which is equivalent for cobounded actions.} If there is a dominating map $f:X \to Y$, we say $X$ \textit{dominates} $Y$ and write $Y \preceq X$.
        \item a \textit{quasi-isometry} if there are constants $A > 1, \ B > 0$ so that $$\frac{1}{A}d_X(x,y) - B \leq d_Y(f(x),f(y)) \leq Ad_X(x,y) + B$$ and if $f(X)$ is \textit{coarsely dense} in $Y$. A subset of a metric space is coarsely dense if there is a constant $C>0$ so that its $C$-neighborhood is the entire space.
        \item a \textit{coarse equivalence of actions} if $f$ is a coarsely $G$-equivariant quasi-isometry.
    \end{enumerate}
\end{definition}

We note that two cobounded actions $G \curvearrowright X$ and $G \curvearrowright Y$ are equivalent if and only if $Y \preceq X$ and $X \preceq Y$, see \cite[Lemma 3.8]{ABO}.

The theory of groups acting coboundedly on metric spaces runs parallel to the theory of groups acting transitively on sets, with the below variant of the Milnor--Schwartz Lemma playing the role of the orbit-stabilizer theorem. For a group $G$ generated by $S = S\inv$, let $\Cay(G,S)$ denote the Cayley graph of $G$ with respect to $S$. 

\begin{lemma}{\cite[Lemma 3.11]{ABO}}
\label{lem:Milnor-Schwartz}
    Let $G$ be a group acting coboundedly on a metric space $X$. Let $B \subset X$ be a subset of diameter $D$ so that $\bigcup_{g \in G}g\cdot B = X$ and let $b \in B$. Then $G$ is generated by the set $$S_{X,B} = \{g \in G \mid d_X(b,g\cdot b) < 2D+1\},$$ and the orbit map $\mathcal{O}:\Cay(G,S_{X,B}) \to X$ sending $g\to g\cdot b$ is a coarse equivalence of actions. 
\end{lemma}

\begin{lemma}{\cite[Lemmas 3.9 and 3.12]{ABO}}
\label{lem:posets_equiv}
    Let $G$ act coboundedly on metric spaces $X$ and $Y$, and let $S_X$ and $S_Y$ be any two generating sets constructed as in the previous lemma. Then $X$ dominates $Y$ if and only if the identity map $\Cay(X, S_X) \to \Cay(G,S_Y)$ is Lipschitz.
\end{lemma}

Then generating sets of $G$, considered up to bilipschitz equivalence of their Cayley graphs exactly parameterize cobounded actions on metric spaces up to coarse equivalence. This motivates the following definition.

\begin{definition}[\cite{ABO}]
\label{def:cobounded_poset}

Let $\mathcal{G}(G)$ denote the set of equivalence classes of generating sets of $G$, where two generating sets $S_1,S_2$ are equivalent if the identity map $\Cay(G,S_1) \to \Cay(G,S_2)$ is bilipschitz. This set is equipped with the relation $\preceq$, where $[S_1] \preceq [S_2]$ if the identity map $\Cay(G,S_1) \to \Cay(G,S_2)$ is Lipschitz. We call $\mathcal{G}(G)$ the \textit{partially ordered set of cobounded actions of $G$}.
\end{definition}

If $G$ is generated by some finite $S = S\inv$, then for every other generating set $T$, the identity map $\Cay(G,S) \to \Cay(G,T)$ is Lipschitz. Then $\mathcal{G}(G)$ has a maximum element whenever $G$ is finitely-generated. In the analogy between cobounded actions on metric spaces and transitive actions on sets, this maximum element is the free action.  

We note one simple fact about coarsely equivariant maps between cobounded actions which will be useful later in the paper. This is an analog of the fact that an equivariant map between two transitive $G$-actions is determined by where it sends any element.

\begin{lemma}[Coarse commutativity]
\label{lem:coarse_commutativity}

Let $G$ act coboundedly on spaces $X$ and $Y$. Then for any two dominations $f,h:X \to Y$, there is a constant $E$ so that $d_Y(f(x),h(x)) < E$. In particular, any two compositions $f_1\circ \cdots f_k$ and $g_1\circ \cdots g_\ell$ of dominations $f_i,g_j$ with the same domain and codomain are equal up to bounded error. 

\end{lemma}

We call this ``coarse commutativity" since it essentially says that any diagram of dominations commutes up to bounded error.

\begin{proof}
    Let $f,h:X \to Y$ be two $D$-coarsely equivariant, $(A,B)$-coarsely Lipschitz maps. Choose $x_0 \in X$ and let $C>0$ be large enough so that for every $x \in X$, there is a $g_x \in G$ so that $d_X(g_x \cdot x_0, x) < C$. We now compute:

    \begin{align*}
        d_Y(f(x),h(x)) &\leq 2(AC +B) + d_Y(f(g_x\cdot x_0), h(g_x\cdot x_0))\\
        &\leq 2(AC+ B) +2D + d_Y(g_x\cdot f(x_0), g_x\cdot h(x_0))\\
        &= 2(AC+ B) +2D + d_Y( f(x_0),  h(x_0))
    \end{align*}

    Then any two such maps are at bounded distance.
\end{proof}

\subsection{Hyperbolic spaces and hyperbolic structures}

In this paper, we are especially concerned with groups acting on Gromov hyperbolic spaces.

\begin{definition}
    Let $X$ be a geodesic metric space. 

    \begin{enumerate}
        \item We say $X$ is (Gromov) \textit{hyperbolic} if there is some $\delta > 0$ so that for any geodesic triangle $\Delta$, each side of $\Delta$ is contained in the union of the closed $\delta$-neighborhoods of the other two sides.
        \item Let $x,y,z \in X$. The \textit{Gromov product} of $x$ and $y$ with respect to $z$ is $$(x,y)_z = \frac{1}{2}(d(x,z) + d(y,z) - d(x,y)).$$
        \item Fix some basepoint $o \in X$. The \textit{(sequential) boundary,} denoted $\partial X$ is the set of equivalence classes of sequences $(x_i)_i \subset X$ of points satisfying $\liminf_{i,j\to\infty} (x_i,x_j)_o = \infty$, where two sequences $(x_i)_i$ and $(y_i)_i$ are equivalent if $\liminf_{i,j \to \infty}(x_i,y_i)_o = \infty$. A basis of neighborhoods of $[\eta] \in \partial X$ is given by $$U_r = \{q \in \partial X \mid \exists(x_i)_i, (y_i)_i \ \text{so that}\ [(x_i)] = [\eta], [(y_i)] = q \ \text{and} \ \liminf_{i,j\to \infty} (x_i,y_i)_o \geq r\}.$$
        \item Fix some basepoint $o \in X$. The \textit{geodesic boundary} of $X$, also denoted $\partial X$, is the set of equivalence classes of geodesic rays starting from $o$, where two rays are equivalent if they are at finite Hausdorff distance. The geodesic boundary is endowed with the compact-open topology.
        \item Let $G$ act on a hyperbolic space $X$. The \textit{limit set} of $G$ is $\Lambda(G):=\partial X \cap \overline{G \cdot x}$.
        \item Let $\infty \in \partial X$. The \textit{parabolic boundary} of $X$ at $\infty$ is $\partial_\infty X = \partial X \setminus\{\infty\}$.
    \end{enumerate}
\end{definition}

It is a standard fact that for proper, geodesic, hyperbolic spaces $X$, the sequential and geodesic boundaries are homeomorphic and independent of basepoint (see \cite[Proposition 2.14]{BoundariesHyperbolicGroups}). For non-proper spaces, the situation is a bit more delicate (see \cite[Remark 2.16]{BoundariesHyperbolicGroups}). We have occasion to use both definitions. For now, we just note that the spaces $X$ and $Y$ in \cref{thm:main} are assumed to be proper, while the hyperbolic structures we define below are not. 

\begin{definition}
\label{def:hyp_poset}
    The \textit{poset of hyperbolic structures} $\mathcal{H}(G)$ is the subposet of $\mathcal{G}(G)$ consisting of the classes of generating sets $[S]$ for which $\Cay(G,S)$ is Gromov hyperbolic. 
\end{definition}

There is a famous classification of group actions on hyperbolic spaces due to Gromov \cite{Gromov}.

\begin{theorem}[Gromov, \cite{Gromov}]
\label{thm:Gromov_Classification}
    Let $G$ act on a hyperbolic space $X$. Then exactly one of the following conditions holds.
    \begin{enumerate}
        \item $|\Lambda(G)| = 0$, or equivalently $G$ has bounded orbits. In this case, the action is called \textit{elliptic}.
        \item $|\Lambda(G)| = 1$. In this case, the action is called \textit{horocyclic}. Horocyclic actions are never cobounded.
        \item $|\Lambda(G)| = 2$. In this case, the action is called \textit{lineal}. 
        \item $|\Lambda(G)| = \infty$ and $G$ fixes a unique point on the boundary. In this case, the action is said to be \textit{focal}, or \textit{quasi-parabolic}.
        \item $|\Lambda(G)| = \infty$ and there is no fixed point on the boundary. In this case, the action is said to be of \textit{general type}. Groups with actions of general type always contain a free group. 
    \end{enumerate}
\end{theorem}

In the particular case of a cyclic group $\langle g \rangle$, only the first three cases can happen, in which case $g$ is called \textit{elliptic}, \textit{parabolic} or \textit{hyperbolic}. If $g$ acts hyperbolically, there is often a unique geodesic connecting the two fixed points at infinity---for instance, this is true if $X$ is geodesically complete and $\CAT(-\kappa)$, as in the hypotheses of \cref{thm:main}. In this case, this unique geodesic is called the \textit{axis} of $g$ and is denoted $\Axis_X(g)$. 

Since horocyclic actions are not cobounded, and the groups we consider are discrete and amenable, $\mathcal{H}(G)$ consists only of elliptic, lineal and quasi-parabolic actions. Of these, the quasi-parabolic actions are of particular interest. The below lemma implies that quasi-parabolic actions of amenable groups always dominate a lineal action.

\begin{proposition}{\cite[Corollary 3.9]{CCMT}}
\label{prop:ccmt_Busemann_function}
    Let $G$ be a locally compact group acting continuously by isometries on a Gromov hyperbolic space $X$ and fixing a boundary point $\xi \in \partial X$. If $G$ is amenable, then there is a continuous homomorphism $\beta_\xi: G \to \R$ satisfying $\beta_\xi(g) \neq 0$ if and only if $g$ acts on $X$ hyperbolically. 
\end{proposition}

The map $\beta_\xi$ is called the \textit{Busemann character} of $G$ at $\xi$, and measures much group elements move points away from $\xi$. In the next section, we will give an explicit description of $\beta_\xi$ with additional assumptions on $X$. For now we just note that if $G$ has a unique lineal structure $L$ (as asserted by Part 1 of \cref{thm:main}), then the action of $G$ on $\R$ via $\beta_\xi$ is equivalent to $L$ and therefore any quasi-parabolic structure must dominate $L$.

\subsection{Horocyclic products}

\cref{thm:main} concerns groups acting geometrically on horocyclic products, which we now define.

Let $X$ and $Y$ be proper, geodesically complete $\CAT(-\kappa)$ spaces. This means that closed balls are compact, every geodesic segment can be extended to a bi-infinite path isometric to $\R$, and $X$ satisfies the $\CAT(-\kappa)$ inequality, below.

\begin{definition}
    Let $X$ be a geodesic metric spaces. Let $\kappa >0$ and let $\h^2_\kappa$ denote the unique complete, simply connected surface of constant curvature $-\kappa$. For any three points $x_1,x_2,x_3\in X$, there are points $X_i \in \h^2_\kappa$ so that $d_X(x_i,x_j) = d_{\h^2_\kappa}(X_i,X_j)$. For $y$ on a geodesic connecting $x_i$ and $x_j$, there is a \textit{comparison point} $Y$ so that $d_X(x_i, y) = d_{\h^2_\kappa}(X_i,Y)$ and $d_X(x_j, y) = d_{\h^2_\kappa}(X_j,Y)$. 
    
    A space $X$ is said to be $\CAT(-\kappa)$ if for all points $x_1,x_2,x_3 \in X$ and $y_1, y_2$ on the geodesic triangle $\Delta(x_1,x_2,x_3)$, the comparison points $Y_1,Y_2 \in \h^2_\kappa $ satisfy $d_X(y_1,y_2)\leq d_{\h^2_\kappa}(Y_1,Y_2)$
\end{definition}

$\CAT(-\kappa)$ spaces are in particular Gromov hyperbolic. 

\begin{definition}[Busemann function]
    Let $X$ be a geodesically complete $\CAT(-\kappa)$ space. Fix a distinguished boundary point $\infty_X \in X$ and let $\gamma$ be a length-parameterized geodesic satisfying $\gamma(\infty) = \infty_X$. A \textit{Busemann function at $\infty_X$} is a function $$\beta(x) = \lim_{t\to \infty} d(\gamma(t),x) - t.$$ A level set of a Busemann function is a \textit{horosphere}. The negative of a Busemann function is called a \textit{height function} $h = -\beta$.
\end{definition}

Re-parameterizing $\gamma$ changes $\beta$ by an additive constant. In a $\CAT(-\kappa)$ space, this is the only possible ambiguity: if $\eta$ is another geodesic satisfying $\eta(\infty) = \gamma(\infty) = \infty_X$, then the associated Busemann functions differ by a constant (see \cite[Lemma 2.1.12]{CaplingerLevitin}). We will call geodesics representing $\infty_X$ \textit{vertical geodesics}. Every vertical geodesic $\eta$ makes the composition $\beta\circ \gamma$ an isometry $\R \to \R$ (see the remark following Definition 2.1.8. in \cite{CaplingerLevitin}).  

\begin{proposition}{\cite[Proposition 2.1.14]{CaplingerLevitin}}
    Let $X$ be $\CAT(-\kappa)$, let $\infty_X \in \partial X$ and let $\beta$ be a Busemann function based at $\infty_X$. Let $G$ be a group acting on $X$ which fixes $\infty_X$. Then $$h(g) = \beta(x) - \beta(g\cdot x)$$ does not depend on $x$, and is a homomorphism $h:G\to \R$. 
\end{proposition}

We are finally prepared to define horocyclic products.

\begin{definition}
    The \textit{horocyclic product} of two $\CAT(-\kappa)$ spaces $X$ and $Y$ equipped with Busemann functions $\beta_X,\beta_Y$ is $$X\bowtie Y = \{(x,y) \mid \beta_X(x) = -\beta_Y(y) \},$$ equipped with a metric defined as follows. Let $N$ be a norm on $\R^2$ so that 1. $N(1,1) = 1$, 2. $N(x,y) \geq \frac{x+y}{2}$ (these two properties make the norm \textit{admissible}) and 3. $N$ is non-decreasing in each factor ($N$ is \textit{monotone}). Let $d_X,d_Y:(X\bowtie Y)^2 \to \R$ denote the distances between the $X$ and $Y$ coordinates of pairs of points in $X\bowtie Y$. The length of a path $\gamma:[0,1]\to X\bowtie Y$ is defined to be $$\ell_N(\gamma) = \sup_{0 = t_0< t_1 < \cdots t_n = 1 } \sum N(d_X(\gamma(t_i),\gamma(t_{i+1})),d_Y(\gamma_i,\gamma_{i+1})).$$ The metric on $X\bowtie Y$ is then defined to be the infimum of the lengths of all paths connecting pairs of points.
\end{definition}

\subsection{Heintze groups, millefeuille spaces and \cite[Theorem B]{CaplingerLevitin}}
\label{subsec:ThmB}

This paper relies strongly on the results of Levitin and the author \cite{CaplingerLevitin}, in particular Theorem B, which says roughly that geometric actions on horocyclic products can be upgraded to actions on ``nice" horocyclic products. The first step in the proof of \cref{thm:main} is to make this upgrade. In this subsection, we define these ``nice" spaces, and explain the statement of \cite[Theorem B]{CaplingerLevitin}.

\begin{definition}[Heintze group]
    A \textit{Heintze group} is a group of the form $N\rtimes_\alpha \R$, where $N$ is a simply connected nilpotent Lie group, $D:\Lie(N) \to \Lie(N)$ is a derivation of its Lie algebra whose eigenvalues have positive real part, and $\alpha$ is the one parameter subgroup of $\Aut(N)$ given by $t \to e^{tD}$. 
\end{definition}

Heintze \cite{Heintze} proved that every homogeneous manifold of negative sectional curvature is isometric to a Heintze group equipped with some left-invariant Riemannian metric and conversely that every Heintze group has some left-invariant metric with negative curvature. In this paper, Heintze groups will always be implicitly endowed with such a metric.

\begin{lemma}
    Let $N \rtimes_{e^{tD}} \R$ be a Heintze group. Then $e^{tD}$ is uniformly expanding in the sense that there is a $\lambda > 1$ so that for all $v \in \Lie(N)$, we have $\|e^{tD}v\| > \lambda^t \|v\|$ for all $t > 0$.
\end{lemma}

\begin{proof}

Differentiate $f(t) = \|e^{tD}v\|^2$: $$f'(t) = (\langle e^{tD}v, e^{tD}v \rangle)' = 2\langle D e^{tD}v, e^{tD}v \rangle.$$ Since $N \rtimes_{e^{tD}} \R$ admits a negatively curved metric, \cite[Proposition 2, Part B]{Heintze}, says that the symmetric part of the derivation $D$ is positive definite. Then there is some $m > 0$ so that $\langle Dw,w\rangle > m \|w\|^2$ for all $w\neq 0$, and in particular $f'(t) \geq 2m\|e^{tD}\| = 2m f(t)$. Then $f(t) \geq f(0) e^{2mt} = \|v\|^2 e^{2mt}$ for $t > 0$, as required.

\end{proof}

For us, ``nice" spaces will be millefeuille spaces, defined below.

\begin{definition}[Millefeuille space]
    Let $X = N\rtimes \R$ be a Heintze group endowed with a left-invariant Riemannian metric of negative curvature and the Busemann function $\beta_X(n,t) = -t$. Let $T_k$ denote the $(k+1)$-regular tree equipped with a Busemann function $\beta_T$. The millefeuille space $X[k]$ is $$X[k] = \{(x,v) \in X\times T_k \mid \beta_X(x) = \beta_T(v)\}.$$

\end{definition}

Note that millefeuille spaces are not horocyclic products, being defined by the equation $\beta_X = \beta_T$ rather than $\beta_X = -\beta_T$. If either $X$ or $T$ is a line (that is, $N$ is trivial or $k = 2$), then $X[k]$ is $T_k$ or $X$ (respectively). When both $N \neq 1$ and $k > 2$, $X[k]$ is said to be a \textit{mixed millefeuille space}. Millefeuille spaces are $\CAT(-\kappa)$, with parabolic boundary homeomorphic to $N\times \Q_k$ (in particular, the parabolic boundary is metrizable). Millefeuille spaces were first introduced in \cite{CCMT}, whose definition was slightly more relaxed, allowing $X$ to be any $\CAT(-\kappa)$ space. 

Theorem B of \cite{CaplingerLevitin} says that if $G$ acts geometrically on a sufficiently regular horocyclic product $X\bowtie Y$, then $X$ and $Y$ may be taken to be millefeuille spaces. This will require the additional assumption that the action is positive. If $X = Y$ and the defining norm is symmetric, there is an isometry of $X\bowtie Y$ given by $(x,y) \to (y,x)$. To be a \textit{positive action} on $X\bowtie Y$ means that these types of isometries are disallowed. More formally, it means that the action fixes the upper and lower boundaries of $X\bowtie Y$---see \cite[Lemma 5.2.1.]{CaplingerLevitin}.

\begin{theorem}{\cite[Theorem B]{CaplingerLevitin}}
    Let $X\bowtie Y$ be a horocyclic product of proper, geodesically-complete $\CAT(-\kappa)$ spaces equipped with Busemann functions about non-isolated points at infinity. Let $G$ be a finitely-generated group acting geometrically and positively on $X\bowtie Y$. Then 
    
    \begin{enumerate}
        \item There are millefeuille spaces $Z,W$ quasi-isometric to $X$ and $Y$ respectively.
        \item There are cobounded actions of $G$ on $Z$ and $W$ satisfying $h_X(g) = h_Y(g)$.
        \item The induced action of $G$ on $Z\bowtie W$ is geometric. In particular, the actions on $X\bowtie Y$ and $Z\bowtie W$ are coarsely equivalent.
    \end{enumerate} 
\end{theorem}

This theorem lets us assume that $X$ and $Y$ are millefeuille spaces, and moreover that the action on $X\bowtie Y$ is by maps of the form $(x,y) \to (f(x),g(y))$ for isometries $f \in \Isom(X)$ and $g \in \Isom(Y)$ fixing their respective points at infinity and satisfying $h_X(f) = -h_Y(g)$. We call such isometries \textit{product maps}. A product map $F = (f,g)$ with non-zero height change will be called \textit{hyperbolic}, and (when it exists) its \textit{axis} is $\Axis(F) = \Axis(f)\bowtie \Axis(g) \subset X\bowtie Y$. The two actions of $G$ on $X$ and $Y$ will end up being the maximum elements of $\mathcal{P}_{\pm}(G)$. 

The group of product maps will be denoted $\Isom(Z)_\infty \times_{-h} \Isom(W)_\infty$, the fibered product of $\Isom(Z)_\infty$ and $\Isom(W)_\infty$ along $h_Z$ and $-h_W$. By composing $G \to \Isom(Z)_\infty \times_{-h} \Isom(W)_\infty$ with the height function on the first factor, we obtain a map $h:G \to \R$. By \cite[Lemma 6.2.3]{CaplingerLevitin}, if $G$ acts geometrically on $X\bowtie Y$, then $h$ has discrete image. Considered as an action on $\R$, the map $h:G \to \Z$ will represent the unique lineal action in \cref{thm:main}. A theorem of Adams \cite[Theorem 6.8]{Adams} implies that $\Isom(Z)_\infty \times_{-h} \Isom(W)_\infty$ is amenable, so $G$ is also amenable (see also \cite[Section 6.1]{CaplingerLevitin}). 

The statement of Theorem B given in \cite{CaplingerLevitin} does not make a positivity assumption, and instead passes to a finite index subgroup. This finite index subgroup is introduced by \cite[Theorem C]{CaplingerLevitin}, which only requires passage to a finite index subgroup when the action is not positive. Restricting to a finite index subgroup changes the possible geometric actions---in particular, one might no longer have actions on $X$ or $Y$ if the action is not positive---so we should not expect a theorem similar to \cref{thm:main} to hold without this condition. Positivity of the action is not used in the paper except to invoke the above theorem.

\section{Unique lineal action}
\label{sec:virt_Cyclic}

In this section, we prove Part 1 of \cref{thm:main}, that a group $G$ acting geometrically on a horocyclic product $X\bowtie Y$ has a unique lineal action. This turns out to be essentially equivalent to showing that $G$ has virtually $\Z$ abelianization, so we spend the majority of this section proving this.

\begin{proposition}
\label{prop:virt_ab}
    Let $G$ act geometrically and positively on a horocyclic product. Then $G\ab$ is virtually $\Z$. 
\end{proposition}

We now prove Part 1 of \cref{thm:main} ($G$ has a unique lineal structure) assuming \cref{prop:virt_ab}. This argument is modeled off of the proof of Proposition 4.1 in \cite{Valuations}.

\begin{proof}[Proof of Part 1 of \cref{thm:main}]

    We first claim that there is a unique orientable lineal structure. A lineal structure is said to be \textit{orientable} if it acts trivially on its boundary. Because $G$ has virtually $\Z$ abelianization, the space of homomorphisms $\Hom(G,\R)$ is one dimensional. Since $G$ is amenable, every pseudocharacter\footnote{A \textit{pseudocharacter} on $G$ is a map $f:G\to \R$ so that there is a constant $D$ so that $|f(gh) - f(g)-f(h)| < D$ for all $g,h \in G$, and so that the restriction of $f$ to every cyclic subgroup is a homomorphism. Psuedocharacters will only be used to invoke \cite[Theorem 4.22]{ABO}.} is a genuine homomorphism, so the space of pseudocharacters is also one dimensional. Then Part b, case 2 of Theorem 4.22 of \cite{ABO} says that there is a unique orientable lineal structure given by the nontrivial homomorphism $h:G \to \Z$ considered as an action on $\R$. The generating set corresponding to this action is $A\cup\{t,t\inv\},$ where $A = \ker(h)$ and $t \in h\inv(1)$.

    Now let $W$ be any (possibly non-orientable) lineal structure. We claim that $A$ acts elliptically (that is, with bounded orbits) on $W$. If not, there is some hyperbolic element $g \in A$. Let $H \subset G$ be the subgroup with index at most two which stabilizes each boundary point and let $h':H \to \R$ be the corresponding Busemann character. As a finite index subgroup, $H\subset G$ also acts geometrically on $X\bowtie Y$, so $\dim(\Hom(H,\R)) = 1$. Then there is some $\lambda \neq 0$ so that, $\lambda h|_H = h'$. But then $h'(g^2) = \lambda h(g^2) = 0$, so $g$ is not hyperbolic. This shows that $A$ contains no hyperbolic elements, and therefore acts elliptically.

    Let $S_W$ be a generating set given by the Milnor--Schwartz Lemma (\cref{lem:Milnor-Schwartz}) so that the orbit map $\Cay(G,S_W) \to W$ is a quasi-isometry. Let $t \in h\inv(1)$. Then clearly $A \cup\{t,t\inv\}$ generate $G$, and since $A$ acts elliptically, $$\sup_{g \in A \cup\{t,t\inv\}} |g|_{S_W} < \infty.$$ Then the identity map $\Cay(G, A \cup \{t,t\inv\}) \to \Cay(G, S_W)$ is Lipschitz, so the lineal structure $[A \cup \{t,t\inv\}]$ dominates the action on $W$. By Theorem 4.22 of \cite{ABO}, two distinct lineal structures are incomparable, so we conclude that they are equal. Then $G$ has a unique lineal structure.

\end{proof}

\subsection{Virtually cyclic abelianization.}

In this subsection, we will prove \cref{prop:virt_ab}, which says that a group acting geometrically and positively on a horocyclic product has virtually cyclic abelianization. The proof has three broad steps: 

\begin{enumerate}
    \item Invoke \cite{CaplingerLevitin}: Use Theorem B of \cite{CaplingerLevitin} to assume that $X$ and $Y$ are Millefeuille spaces, and that the action on $X\bowtie Y$ is by product maps. Since $X$ and $Y$ are Millefeuille spaces, we also obtain an action on a Diestel-Leader graph $\DL(n,m)$. If $v \in \DL(n,m)$ is any vertex, then $\Stab_G(v)$ acts geometrically on $\pi\inv(v)$, a simply connected nilpotent Lie group.
    \item Assume for the sake of contradiction that $G\ab$ is not virtually $\Z$. Then there is a surjection $k:G\to \Z$ linearly independent from $h$. By replacing $k$ with a linear combination of $k$ and $h$, we may assume that there is an element $g \in \ker(h)$ with $k(g)\neq 0$. Taking powers of $g$ shows that for every $v \in \DL(n,m)$, $\Stab_G(v)$ contains an element not in $\ker(k)$.
    \item Let $t \in G$ be hyperbolic and let $v \in \Axis_{\DL(n,m)}(t)$. Then $t$ commensurates $\Stab_G(v)$, and moreover acts like an Anosov map (in a sense to be made precise). We then use Malcev rigidity to show that $k(tgt\inv) = k(g)$ implies $k$ is trivial\footnote{Compare to the elementary linear algebra fact that if $\alpha:\R^n\to \R^n$ has eigenvalues off the unit circle, and $K:\R^n \to \R$ satisfies $K\circ \alpha = K$, then $K$ is trivial.} on $\Stab_G(v)$, contradicting the previous paragraph.

\end{enumerate}

\noindent \textbf{Step 1.} Let $G$ act geometrically and positively on $X\bowtie Y$. Theorem B of \cite{CaplingerLevitin} gives millefeuille spaces $Z,W$ quasi-isometric to $X$ and $Y$ respectively, and actions of $G$ on $Z$ and $W$ so that the induced product action on $Z\bowtie W$ is geometric. We proceed assuming that $X$ and $Y$ are millefeuille spaces and that $G$ acts by products of isometries. To fix notation, let $N_1,N_2$ be simply connected nilpotent groups, $\alpha_i: \R \to \Aut(N_i)$ be expanding one parameter families of automorphisms of $N_i$ and $n,m \in \N$ so that $X = (N_1\rtimes_{\alpha_1} \R)[n]$ and $Y = (N_2 \rtimes_{\alpha_2} \R)[m]$. We now explicitly compute the group of product isometries of $X\bowtie Y$.

\begin{lemma}
\label{lem:heintzeIsom}
    Let $N$ be a simply connected nilpotent Lie group and $\alpha_t \in \Aut(N)$ a one parameter family of automorphisms. Give $N$ a left-invariant Riemannian metric and give $N\rtimes_\alpha \R$ the left-invariant Riemannian metric which restricts to the given one on $N$ and makes the $\R$ factor orthogonal to $N$. Let $\Isom(N\rtimes_\alpha \R)_\infty$ denote the isometries which induce a translation on the coset partition of $N$. Let $\AutIsom(N)$ denote the automorphisms of $N$ which are also isometries. Then $$\Isom(N\rtimes_\alpha \R)_\infty = \left( N\rtimes \bigcap_{s\in \R} \alpha_s\AutIsom(N) \alpha_s\inv  \right)\rtimes \R,$$ where the action of $\bigcap_{s\in \R} \alpha_s\AutIsom(N) \alpha_s\inv$ on $N$ is the obvious one and the action of $\R$ is by $s\cdot (n, A) = (\alpha_s(n),\alpha_s A \alpha_s\inv)$. 
    
\end{lemma}

\begin{proof}

    Let $h:\Isom(N\rtimes_\alpha \R) \to \R$ denote the induced action on heights. Since left translation by $(1,s)$ is an isometry, this height function splits. The kernel of $h$ consists of all isometries of $N$ which extend to isometries of $N\rtimes_\alpha \R$. By \cite{Isom_Nilpotent}, $\Isom(N) = N \rtimes \AutIsom(N)$. Clearly left translation in $N$ extends to an isometry of $N\rtimes_\alpha \R$. For $A \in \AutIsom(N)$, we show that $F_A:N\rtimes_\alpha \R \to N\rtimes_\alpha \R$ given by $(n,s) \to (A(n),s)$ is an isometry of $N\rtimes_\alpha \R$ if and only if $\alpha_s A \alpha_s\inv \in \AutIsom(N)$ for all $s$.

    Let $v_1,\ldots v_n$ be an orthonormal basis of the Lie algebra $\Lie(N)$, and let $\partial_t$ be the tangent vector at $(1,0) \in N\rtimes \R$ in the $\R$ direction. Let $(n,s) \in N\rtimes_\alpha \R$ and let $L_{(n,s)}$ denote left multiplication by $(n,s)$. Then the vectors $d(L_{(n,s)})_e(v_i)$ and $\partial_t$ form an orthonormal basis of $T_{(n,s)}(N\rtimes_\alpha \R)$. Assume that $\alpha_sA \alpha_s\inv \in \AutIsom(N)$ for all $s \in \R$, and set $B_s = \alpha_s\inv A \alpha_s $. A straightforward computation shows that $$F_A \circ L_{(1,s)} = L_{(1,s)} \circ F_{B_s} \quad \text{and} \quad F_A\circ L_{(n,0)} = L_{(A(n),0)} \circ F_A.$$ Together, these imply $F_A\circ L_{(n,s)} = L_{(A(n),s)}\circ F_{B_s}$, so that $$dF_A\circ dL_{(n,s)}(v_i) = dL_{(A(n),s)}\circ dF_{B_s}(v_i).$$ Since $B_s \in \Isom(N)$, the vectors $dL_{(A(n),s)}\circ dF_{B_s}(v_i)$ form an orthonormal basis of the horizontal subspace at $(A(n),s)$. Since $dF_A(\partial_t) = \partial_t$ also, we conclude that $F_A$ is an isometry.

    Conversely, let $n \in N$ and $s \in \R$ and compute $$L_{(1,s)}(L_{(n,0)} F_A) L_{(1,s)}\inv = L_{(\alpha_s(n),0)}(L_{(1,s)} F_A L_{(1,s)}\inv) = L_{(\alpha_s(n),0)} F_{\alpha_s A\alpha_s\inv}.$$ In particular, if $F_A$ is an isometry, then $F_{\alpha_s A\alpha_s\inv}$ is also. This shows that the kernel of $h$ is indeed $N\rtimes \bigcap_{s\in \R}\alpha_s \AutIsom(N) \alpha_s\inv$. The above computation also shows that the action of $\R$ on $\ker(h)$ is $s\cdot (n,A) = (\alpha_s(n), \alpha_s A\alpha_s\inv)$, proving the lemma.
    
\end{proof}

We can now give a precise description of the group of product isometries of $X\bowtie Y = (N_1\rtimes_{\alpha_1}\R)[n] \bowtie (N_2\rtimes_{\alpha_2} \R)[m].$

\begin{lemma}
\label{lem:bowtieMillefeuilleIsom}

    The group of product isometries of $(N_1\rtimes_{\alpha_1} \R)[n] \bowtie (N_2\rtimes_{\alpha_2} \R)[m]$ is $$(H_1 \times_h \Aut(T_n)) \times_{-h} (H_2 \times_h\Aut(T_m)),$$ where $$H_i = \left( N_i\rtimes \bigcap_{s\in \R} (\alpha_i)_s\AutIsom(N) (\alpha_i)_s\inv  \right)\rtimes \R,$$ and $\Aut(T_k)$ denotes the automorphisms of the regular $(k+1)$-valent tree which fix a distinguished point at infinity. The notation $\times_h$ denotes the fibered product along height functions and $\times_{-h}$ denotes the the fibered product of the height function on one group and the negative of the height function on the other.

\end{lemma}

\begin{proof}
    This lemma is a direct combination of \cref{lem:heintzeIsom} and \cite[Lemma 7.2.1.]{CaplingerLevitin}, which says that $\Isom(M[k])\cong \Isom(M)_\infty \times_{h} \Aut(T_k)$.
\end{proof}

In particular, the group of product isometries projects to $\Aut(T_n)\times_{-h} \Aut(T_m)$, and therefore acts on the Diestel-Leader graph $\DL(n,m) = T_n \bowtie T_m$, making the projection $\pi: X\bowtie Y \to \DL(n,m)$ equivariant. The following Lemma is a generalization of \cite[7.2.3]{CaplingerLevitin}.

\begin{lemma}
\label{lem:StabGeometric}
    Let $v_0 \in \DL(n,m)$ and assume that $G$ acts geometrically by product isometries on $X\bowtie Y$. Then $\Stab_G(v_0)$ acts geometrically on $\pi\inv(v_0) \cong N_1 \times N_2$.
\end{lemma}

\begin{proof}
    The proper discontinuity of the action of $\Stab_G(v_0)$ on $\pi\inv(v_0)$ follows directly from the proper discontinuity of $G$ on $X\bowtie Y$. It remains to show that the action is cocompact.

    Let $K \subset X\bowtie Y$ be a compact set whose $G$-translates cover $X\bowtie Y$. Then $\pi(K)$ is compact and therefore contains finitely many vertices. For each of the finitely many vertices $w \in G\cdot v_0 \cap \pi(K)$, choose a $g_w \in G$ so that $g_w \cdot w = v_0$. Set $$C = \left(\bigcup_{w \in G\cdot v_0 \cap \pi(K)} g_w \cdot K \right) \cap \pi\inv(v_0).$$ This is clearly compact. Let $x \in \pi\inv(v_0).$ Then there is some $g \in G$ so that $g\cdot x \in K$. Let $w = \pi(g\cdot x)$. Then $g_wg \cdot x \in C$, and we conclude that the $G$-translates of $C$ cover $\pi\inv(v_0)$.
\end{proof}

Let $N = N_1\times N_2$ and $\alpha(t) = (\alpha_1(t),\alpha_2(-t))$ and let $p: G \to \Isom(N \rtimes_ \alpha \R)$ denote the projection to $H_1 \times_{-h} H_2 \subset \Isom(N \rtimes_ \alpha \R)$ given by \cref{lem:bowtieMillefeuilleIsom}. Lemma \ref{lem:StabGeometric} says that for all vertices $v \in \DL(n,m)$, the stabilizer $\Stab_G(v)$ acts geometrically on $\pi\inv(v)$. Then $p(\Stab_G(v))$ is a discrete an cocompact subgroup of $N \rtimes \AutIsom(N)$. A theorem of Auslander \cite{Auslander} now implies that $\Gamma = N \cap p(\Stab_G(v))\subset N \rtimes \AutIsom(N)$ is a lattice in $N$ and a finite index subgroup of $p(\Stab_G(v))$.

We note also that the action on $\DL(n,m)$ gives a height-change homomorphism $h:G\to \Z$. Fix an element $t \in h\inv([1,\infty))$.

\vspace{.3cm}

\noindent \textbf{Step 2.} Assume for the sake of contradiction that $G\ab$ is not virtually $\Z$. Then there is some surjection $k':G \to \Z$ which is linearly independent from $h$. Fix an element $t \in h\inv([1,\infty))$ and set $k = h(t)k' - k'(t) h$, a non-trivial homomorphism satisfying $k(t) = 0$. If $k(\ker(h)) = 0$ also, then $k$ would be trivial, so there is an element $g \in \ker(h)$ satisfying $k(g) \neq 0$. Every elliptic element of $\Aut(T_n)$ and $\Aut(T_m)$ fixes a vertex in each tree, so some power of $g$ fixes a vertex in $\DL(n,m)$. By taking further powers, we can produce elements fixing arbitrarially large balls in $\DL(n,m)$. In particular, for each $v \in \DL(n,m)$ there is an element $g^N \in \Stab_G(v)$ not in $\ker(k)$. Then $k$ restricted to each vertex stabilizer is non-trivial. This last statement is what we will contradict in step 3.

If $X$ and $Y$ are both trees, then no element in a vertex stabilizer can have infinite order, since then the action would fail to be properly discontinuous. In that case, we have already contradicted the previous paragraph. We proceed assuming that at least one of $N_1, N_2$ is non-trivial.

\vspace{.3cm}

\noindent \textbf{Step 3.} By Lemma \ref{lem:bowtieMillefeuilleIsom}, we know that $p(t)$ takes the form $L_{(n,0)}  F_A  L_{(1,s)}$, where $n \in N = N_1\times N_2$ and $A \in \AutIsom(N_1)\times \AutIsom(N_2) \subset \AutIsom(N)$ is an isometric automorphism which satisfies $\alpha_r A \alpha_r\inv \in \AutIsom(N)$ for all $r \in \R$. By conjugating the entire action, we may assume that the axis of $t$ passes through the identity element of $N\cong \pi\inv(v) \subset X\bowtie Y$, and therefore that $n = 1$.

Choose a vertex $v \in \Axis_{\DL(n,m)}(t)$. From Step 1, $p(\Stab_G(v)) \subset N\rtimes \AutIsom(N)$ contains a finite index subgroup $\Gamma = N \cap p(\Stab_G(v))$ which is a lattice in $N$. There is a further finite index subgroup $\Lambda \subset \Gamma$ so that $p|_{\Stab_G(v)}\inv (\Lambda)$ fixes the $h(t)$-ball in $\DL(n,m)$ centered at $v$. Then $t p|_{\Stab_G(v)}\inv(\Lambda) t\inv \subset \Stab_G(v)$. 

We now compute the action of conjugation by $p(t)$ on $N \subset N\rtimes \R$: $$(F_AL_{(1,s)}) L_{(m,0)} (F_AL_{(1,s)})\inv = F_A L_{(\alpha_s(m),0)} F_A\inv = L_{(A\alpha_s(m))}.$$ In particular, $p(t) \Lambda p(t)\inv \subset \Gamma$. Since the actions of $A$ and $\alpha_s$ split as diagonal actions on $N_1$ and $N_2$, the action of $\alpha_s$ is uniformly expanding on one and uniformly contracting on the other, and $A$ is an isometry, we conclude that the action of conjugation by $p(t)$ on the Lie algebra $\Lie(N)$ has no eigenvalues on the unit circle. 

To summarize: there are lattices $\Lambda \subset \Gamma \subset N = N_1\times N_2$ so that $p|_{\Stab_G(v)}\inv (\Lambda) \subset \Stab_G(v)$ has finite index, and the action of $p(t)$ by conjugation sends $\Lambda$ into $\Gamma$ and acts on $\Lie(N)$ with eigenvalues off the unit circle. Since $\ker(p|_{\Stab_G(v)})$ is finite (by \cref{lem:StabGeometric}), $k:\Stab_G(v) \to \Z$ descends to $p(\Stab_G(v))$, and thereby gives a map on the lattices. 

We now show that the above conditions force $k$ to be trivial.

\begin{proposition}
    \label{prop:k_vanishes}
    Let $N$ be a simply connected nilpotent Lie group and let $\Lambda \subset \Gamma \subset N$ be lattices. Let $\beta \in \Aut(N) = \Aut(\Lie(N))$ be an automorphism with eigenvalues off the unit circle and satisfying $\beta(\Lambda) \subset \Gamma$. Let $k:\Gamma \to \Z$ be a homomorphism satisfying $k(\beta (\lambda)) = k (\lambda)$ for all $\lambda \in \Lambda$. Then $k$ is trivial.  
\end{proposition}

This proposition completes the proof of \cref{prop:virt_ab} (with $\beta = A \alpha_s$), since it implies that $k$ is trivial on $\Stab_G(v)$, contradicting the conclusion of step 2.

\begin{proof}

    By Malcev rigidity (see \cite[Theorem 2.11]{Raghunathan}), the map $k:\Lambda \to \Z$ extends uniquely to a map $K: N \to \R$. Uniqueness and the identity $k\circ \beta = k$ together imply that $K\circ \beta = K$ and therefore that $dK \circ d\beta = dK$. In particular, $d\beta$ leaves every coset of $\ker (dK)$ invariant. Assume that $K$ is non-trivial, so that there is some $v \in \Lie(N)$ satisfying $dK(v) = 1$. Choose a basis $w_1,\ldots, w_k$ of $\ker(dK)$. Then in the basis $w_1,\ldots,w_k, v$ the matrix of $d\beta$ is $$\begin{pmatrix}
  \quad \Asterisk  \quad & \Asterisk \\[1.5ex]
  \mathbf{0} & 1
\end{pmatrix}$$ and hence has 1 as an eigenvalue. But $d\beta$ is assumed to have no eigenvalues on the unit circle. 
    
\end{proof}

This concludes the proof of Part 1 of \cref{thm:main}. We now prove Part 2. This proof is nearly identical to the proof of the corresponding statement in \cite[Proposition 4.1]{Valuations}.

\begin{proof}[Proof of Part 2, \cref{thm:main}]

    Let $G$ act geometrically on $X\bowtie Y$, with our standing assumptions. Since $G$ is amenable---see \cite[Section 6]{CaplingerLevitin}---it admits no hyperbolic actions of general type. Then every hyperbolic structure is either elliptic, lineal or quasi-parabolic. By Part 1 of \cref{thm:main}, there is a unique lineal structure given by some homomorphism $h:G\to \Z$. Let $t \in G$ be any element so that $h (t) > 0$. Since $h$ is the unique nontrivial projective class of homomorphism $G \to \Z$ and Busemann characters detect hyperbolicity (See \cref{prop:ccmt_Busemann_function}), $t$ acts hyperbolically on every quasi-parabolic structure. By definition, every quasi-parabolic structure $Z$ has a unique global fixed point on $\partial Z$, which is either the attracting or repelling fixed point of $t$. Let $\mathcal{P}_+(G)$ denote the hyperbolic structures which fix the attracting fixed point of $t$, and let $\mathcal{P}_-(G)$ denote the hyperbolic structures which fix the repelling fixed point of $t$. Then $\mathcal{H}(G) = \{\ast, L\} \cup \mathcal{P}_+(G), \mathcal{P}_-(G)$. 

    If any element $Z \in \mathcal{P}_+(G)$ dominates $W \in \mathcal{P}_-(G)$, then $G$ fixes both the attracting and repelling fixed points of $t$ on $\partial W$. Then $W$ is lineal, hence equivalent to $L$. This shows that $\mathcal{P}_+(G) \cap \mathcal{P}_-(G) = \{L\}$, and that elements of $\mathcal{P}_+(G) \setminus \{L\}$ are incomparable to elements of $\mathcal{P}_+(G) \setminus \{L\}$.

\end{proof}

\section{Maximum elements.}

In this section, we will prove Part 3 of \cref{thm:main}, that the subposets $\mathcal{P}_{\pm}(G)$ have maximum elements. As in Section \ref{sec:virt_Cyclic}, we use \cite{CaplingerLevitin} to assume that $X$ and $Y$ are millefeuille spaces and that the $G$-action is by product maps. These maximum elements of $\mathcal{P}_{\pm}(G)$ will be represented by the actions on $X$ and $Y$. 

The proof requires a new boundary object associated to a horocyclic product, which we call the \textit{vertical boundary}.

\begin{definition}
    The \textit{vertical boundary} of a horocyclic product, denoted $\mathcal{V}(X\bowtie Y)$ is the space of all height-parameterized vertical geodesics with the compact-open topology.
\end{definition}

The vertical boundary is distinct from the visual boundary studied by Ferragut and from the upper and lower boundaries of a horocyclic product. Levitin and the author showed that the vertical boundary is in bijection with the product of the parabolic boundaries $\partial_\infty X \times \partial_\infty Y,$ see the discussion following Definition 2.4.2 in \cite{CaplingerLevitin}.

We now outline the proof of Part 3 of \cref{thm:main}.

\begin{enumerate}
    \item Verify that the bijection $\mathcal{V} = \mathcal{V}(X\bowtie Y) \cong \partial_\infty X \times \partial_\infty Y$ is in fact a homeomorphism.
    \item Let $Z \in \mathcal{P}_+(G)$ dominate $X$, and let $\pi_Z:X\bowtie Y \to Z$ be a dominating map. We show that $\pi_Z$ induces a continuous, $G$-equivariant map $\partial\pi_Z: \mathcal{V} \to \partial_\infty Z$ by sending a vertical geodesic $\gamma$ to the quasi-geodesic $\pi_Z(\gamma)$.
    \item Hyperbolic elements of $G$ act on $\mathcal{V} \cong \partial_\infty X \times \partial_\infty Y$ diagonally by expanding one of $\partial_\infty X$ and $\partial_\infty Y$ and contracting the other. Since $Z \in \mathcal{P}_+(G)$, a particular hyperbolic element $g\in G$ expands or contracts $\partial_\infty Z$ according to whether it expands or contracts $\partial_\infty X$. We use this fact to show that $\partial \pi_Z: \mathcal{V} \cong \partial_\infty X \times \partial_\infty Y \to \partial_\infty Z$ factors through the projection onto $\partial_\infty X$.
    \item Use the previous point to construct a quasi-inverse to the coarsely equivariant map $Z \to X$. Since the poset $\mathcal{P}_+(G)$ has supremums (by Lemmas 4.3, 4.4 and 4.5 of \cite{Valuations}), this shows that $X$ is the maximum element. 
    
\end{enumerate}

\begin{lemma}
    Let $X$ and $Y$ be millefeuille spaces and let $\mathcal{V} = \mathcal{V}(X\bowtie Y)$ denote the space of height-parameterized vertical geodesics in $X\bowtie Y$ with the compact-open topology. Then $\mathcal{V}$ is homeomorphic to $\partial_\infty X \times \partial_\infty Y$. 
\end{lemma}

\begin{proof}
    Every vertical geodesic in $X\bowtie Y$ projects to a vertical geodesic in the two factors, and conversely two geodesics $V,W$ in the factors give a vertical geodesic $V\bowtie W \subset X\bowtie Y$. We show that this bijection $f:\mathcal{V} \to \partial_\infty X\times \partial_\infty Y$ is in fact a homeomorphism.

    By \cite[Theorem 4.4.2.]{TopPropContFuncSpaces} and the following remark, for topological spaces $Z,W$, the space of continouous function $W\to Z$, $C(W,Z)$ with the compact-open topology is metrizable if and only if $W$ is hemicompact. In particular, $\mathcal{V}\subset C(\R,X\bowtie Y)$ is a sequential space. Since $X$ and $Y$ are both millefeuille spaces, the parabolic boundaries $\partial_\infty X$ and $\partial_\infty Y$ are metrizable. Then it suffices to show that a sequence $V_i\bowtie W_i \in \mathcal{V}$ converges if and only if $V_i\subset X$ and $W_i\subset Y$ also converges. 

    If $V_i\bowtie W_i \in \mathcal{V}$ converge uniformly on compact sets in $X\bowtie Y$, then so must their projections $V_i$ and $W_i$, since the projection onto each factor is distance-decreasing. Conversely, let $V_i \to V$ in $X$ and $W_i \to W$ in $Y$ uniformly on compact sets. Lemma 2.2.3. of \cite{CaplingerLevitin} says that $d_{X\bowtie Y} \leq d_X + d_Y + \min (d_X, d_Y)$. Then $V_i \bowtie W_i$ also converges to $V \bowtie W$ uniformly on compact sets. This completes the proof that $\mathcal{V}$ and $\partial_\infty X\times \partial_\infty Y$ are homeomorphic.
    
\end{proof}

\subsection{Maps on the vertical boundary}

Let $G$ act geometrically on $X\bowtie Y$ and suppose $G$ is generated by a finite set $S = S\inv$. Then $\Cay(G,S)$ is quasi-isometric to $X\bowtie Y$ by the Milnor-Schwartz lemma (\cref{lem:Milnor-Schwartz}). If $G$ also acts coboundedly on a metric space $Z$, then by \cref{lem:posets_equiv} and the remark following \cref{def:hyp_poset}, there is a dominating map $\pi_Z:X\bowtie Y \to Z$. 

\begin{lemma}
    \label{lem:vert_to_quasi}
    Let $G$ act geometrically on $X\bowtie Y$ and coboundedly on a hyperbolic space $Z$. Assume that the action on $Z$ is focal. Then the coarsely equivariant map $\pi_Z: X\bowtie Y \to Z$ sends vertical geodesics $\gamma$ in $X\bowtie Y$ to $(K,C)$-quasi-geodesics $\pi_Z\gamma$ in $Z$, with $(K,C)$ not depending on $\gamma$.
\end{lemma}

\begin{proof}
    Let $\gamma: \R \to X\bowtie Y$ be a vertical geodesic. Since $\pi_Z$ is coarsely Lipschitz, there are constants $A>1, \ B>0$ so that \begin{equation}\label{eqn:upperBound}d(\pi_Z (\gamma_t), \pi_Z(\gamma_s))\leq A |t-s| + B \end{equation} for all $t,s\in\R$. Let $h_Z: Z \to \R$ be the coarsely equivariant, coarsely Lipschitz map to $\R$ (unique up to finite distance by \cref{lem:coarse_commutativity}). Then there are constants $a > 1, \ b >0$ so that $$d_\R(h_Z (\pi_Z \gamma_t), h_Z(\pi_Z\gamma_s)) \leq a \cdot d(\pi_Z(\gamma_t), \pi_Z(\gamma_s)) + b.$$ By coarse commutativity, (\cref{lem:coarse_commutativity}) $d_\infty(h_Z\circ \pi_Z, h) < \infty$. Since $\gamma$ is vertical (that is, $h(\gamma_t) = t$), there is some $C>0$ so that $|h_Z (\pi_Z \gamma_t) - t| \leq C.$ Then $$|t-s| \leq a \cdot d(\pi_Z(\gamma_t), \pi_Z(\gamma_s)) + b + 2C,$$ and therefore \begin{equation}\label{eqn:lowerBound}\frac{1}{a} |t-s| - \frac{1}{a}(b + 2C) \leq d(\pi_Z \gamma_t, \pi_Z \gamma_s).\end{equation} From Equations \ref{eqn:upperBound} and \ref{eqn:lowerBound}, $\pi_Z\gamma$ is a quasi-geodesic in $Z$.

\end{proof}

We next argue that one of the endpoints of $\pi_Z\gamma$ is the unique fixed point $\Fix_{\partial Z}(G):= \infty$, thereby giving a map $\partial \pi_Z: \mathcal{V} \to \partial_\infty Z$ by taking the other endpoint of $\pi_Z \gamma$. This will follow from the fact that $h_Z(\pi_Z\gamma(\R))$ is coarsely dense in $\R$ (by coarse commutativity). 



\begin{lemma}
    We adopt the above setup: Let $G$ act geometrically on $X\bowtie Y$, and coboundedly, quasi-parabolically on $Z$ with fixed point $\Fix_{\partial Z}(G) = \xi$. Let $h:G \to \R$ be the height function, representing the unique lineal structure. Let $h_Z:Z\to \R$ be a coarsely Lipschitz, coarsely equivariant map, and let $\eta$ be a quasi-geodesic in $\Z$ whose image under $h_Z$ is coarsely dense in $\R$. Then one of the endpoints of $\eta$ is $\xi$. 
    
\end{lemma}

The idea of the proof is that $h_Z$ is, up to bounded distance and scaling, a coarse Busemann function about $\xi$, so quasi-geodesic $\eta$ satisfying $h_Z(\eta(t))\to \infty$ must approach $\xi$.

\begin{proof}
    A \textit{horokernel} is an accumulation point in the topology of pointwise convergence of a sequence of functions $$f_n:Z\times Z \to \R \qquad f_n(x,y) = d(x,x_n) - d(y,x_n),$$ where $(x_n)_{n\in \N}$ is a sequence converging to $\xi$. Let $b = \lim_{\alpha \in A} f_{n_\alpha} $ be some horokernel, written as a limit of a subnet of $(f_n)_{n \in \N}$ defined using  some sequence $x_n$ approaching $\xi$. Fix a basepoint $z_0$ and set $h_Z'(z) = b(z_0,z)$. Also let $S_Z \subset G$ be a generating set making the orbit map $\mathcal{O}:\Cay(G,S_Z)\to \Z$ given by $g\to g\cdot z_0$ a quasi-isometry, and let $\mathcal{O}\inv$ denote a coarse inverse.

    Let $\beta:G \to \R$ be the Busemann character associated to $\xi$. Since $G\ab$ is virtually $\Z$, there is some $\lambda \neq 0$ so that $\beta = \lambda h$. By Propositions 3.7 and 3.9 of \cite{CCMT}, the Busemann character $\beta$ is at bounded distance from $g \to b(z_0, g\cdot z_0)$. This is exactly the map $h_Z'\circ \mathcal{O}$. Then $h_Z'$ is a bounded distance from $\beta \circ \mathcal{O}\inv = \lambda h \circ \mathcal{O}\inv$, hence $\frac{1}{\lambda} h_Z'$ is a bounded distance from $h\circ \mathcal{O}\inv$, a domination. By \cref{lem:coarse_commutativity}, $\frac{1}{\lambda}h_Z'$ is a bounded distance from $h_Z$, so we may take $h_Z = \frac{1}{\lambda} h_Z'$. 

    Let $\eta$ be a quasi-geodesic in $Z$ with $h_Z(\eta)$ coarsely dense in $\R$. Then there is a sequence $t_i \in \R$ approaching either $\pm \infty$ so that $h_Z'(\eta(t_i)) \to \infty$. We will now show that $\eta(t_i)$ limits to $\xi$. We now compute \begin{align*}
        2 (\eta(t_i) \mid x_n)_{z_0} & = d(z_0, \eta(t_i)) + d(z_0, x_n) - d(\eta(t_i), x_n)\\
        &\geq d(z_0, x_n) - d(\eta(t_i),x_n).
    \end{align*}

    The limit of the last expression along the subnet is exactly $h_Z'(\eta(t_i))$, so $$\lim_{i\to \infty}\left (\lim_{\alpha \in A}  (\eta(t_i) \mid x_{n_\alpha})_{z_0}    \right) = \infty.$$ By possibly passing to a subsequence of the $\eta(t_i)$, we may assume that $$\lim_{\alpha\in A}(\eta(t_i)\mid x_{n_\alpha})_{z_0}> i+\delta.$$

    For each $i,$ choose $\alpha(i) \in A$ so that $n_{\alpha(i)} \geq i$ and $(\eta(t_i) \mid x_{n_{\alpha(i)}})_{z_0} \geq i + \delta$. Let $M > 0$. Then there is some $N$ so that if $k,j > N$, then $(x_k \mid x_j)_{z_0} > M + \delta$. Now, if $i,j > \max(N,M)$, then $n_{\alpha(i)} \geq N$, so $$(\eta(t_i) \mid x_{n_{\alpha(i)}})_{z_0} \geq i + \delta \geq M+ \delta \quad \text{and} \quad (x_{n_{\alpha(i)}} \mid x_j)_{z_0} \geq M + \delta.$$ Then $$(\eta(t_i) \mid x_j)_{z_0} \geq \min\big( (\eta(t_i) \mid x_{n_{\alpha(i)}})_{z_0} , (x_{n_{\alpha(i)}} \mid x_j)_{z_0}\big) - \delta \geq M.$$ So $\lim_{i,j\to \infty} (\eta(t_i) \mid x_j)_{z_0} = \infty$, meaning the sequences $\eta(t_i)$ and $x_n$ represent the same points at infinity, namely $\xi$.

\end{proof}

There is therefore a function $\partial \pi_Z: \mathcal{V} \to \partial_\infty Z$ defined by sending a vertical geodesic $\gamma$ to the end of $\pi_Z(\gamma)$ not equal to the unique fixed point $\Fix_{\partial Z}(G)$. We now show that this induced boundary map is continuous.


\begin{lemma}
     Let $G$ act geometrically on $X\bowtie Y$ and coboundedly on a hyperbolic space $Z$. Assume that the action on $Z$ is focal with fixed point $\infty \in \partial Z$. Then the induced map $\partial \pi_Z : \mathcal{V}\to \partial_\infty Z$ is continuous.
\end{lemma}

\begin{proof}
    Let $\gamma_i$ be a sequence of vertical geodesics in $X\bowtie Y$ which converge uniformly on compact sets to some $\eta$. Then for all $\epsilon > 0$ and all $M > 0$, there is an $N$ so that if $n > N$, then $d(\gamma_n(t), \eta(t)) < \epsilon$ for all $t \in [-M,M]$. Let $A>1, \ B>0$ be the coarsely Lipschitz constants, so that $$d(\pi_Z \gamma_n(t), \pi_Z \eta(t)) \leq A \epsilon + B.$$ Choose $\epsilon = 1/A$. Then for all $M >0$, there is an $N$ so that if $n \geq N$, the quasi-geodesics $\pi_Z\gamma_n$ and $\pi_Z\eta$ $(B+1)$-fellow-travel for $M$ time. In the following lemma, we show that this is sufficient to conclude that $[\pi_Z\gamma_n] \to [\pi_Z\eta]$ in the boundary $\partial_\infty Z$.
    
\end{proof}

The author was not able to find a reference to the below fact in the required generality. Similar statements appear in several texts with the additional assumption that $Z$ is proper, in which case one has access to the Arzela-Ascoli theorem. We wish to avoid this assumption, so we are forced to use Gromov products.

\begin{lemma}
    Let $\gamma_n$ be a sequence of $(K,C)$ quasi-geodesic rays in a hyperbolic space and $\eta$ another $(K,C)$ quasi-geodesic. Assume that there is some $B>0$ so that for all $M > 0$, there is an $N$ so that every $n \geq N$ satisfies $d(\gamma_n(t), \eta (t)) < B$ for all $t \in [0,M]$. Then $[\gamma_i] \to [\eta]$ in $\partial_\infty Z$.
    
\end{lemma}

\begin{proof}

    A basis of neighborhoods of $[\eta]$ is given by $$U_r = \{q \in \partial X \mid \exists(x_i)_i, (y_i)_i \ \text{so that}\ [(x_i)] = [\eta], [(y_i)] = q \ \text{and} \ \liminf_{i,j\to \infty} (x_i,y_i)_o \geq r\},$$ where $(x,y)_o$ is the Gromov product based at $o$ and we represent boundary points with sequences $(x_i)$ satisfying $\liminf_{i,j \to \infty} (x_i,x_j)_o$ (See \cite[Definition 2.13]{BoundariesHyperbolicGroups}). Recall that the boundary point corresponding to a quasi-geodesic ray $\alpha$ is represented by the sequence $(\alpha(i))_i$. Then it suffices to show that for each $r > 0$, there is an $N$ so that for all $n \geq N$, we have $\liminf_{i,j\to \infty}(\gamma_n(i),\eta(j))_o \geq r$. 

    Let $M < t,s$ and $\gamma, \eta$ be $(K,C)$ quasi-geodesic rays (which we will later specialize to the above $\gamma_n$ and $\eta$). We apply the four points condition twice, first with $\gamma(M)$, then with $\eta(M)$.  
\begin{align*}
        (\gamma(t), \eta(s))_o &\geq \min\{(\gamma(t), \gamma(M))_o, (\gamma(M),\eta(s))_o\} - \delta\\
        &\geq \min\{(\gamma(t), \gamma(M))_o,(\gamma(M),\eta(M))_o, (\eta(M),\eta(s))_o\} - 2\delta
    \end{align*}

    We now bound the first terms of the minimum. The same argument and bound works for the third term. Let $D$ be the morse constant for $(K,C)$ quasi-geodesics, and let $p$ be a point on the geodesic connecting $\gamma(0)$ and $\gamma(t)$ so that $d(p, \gamma(M))\leq D$. Using the reverse triangle inequality and the identity $d(\gamma(0), p) + d(p,\gamma(t)) = d(\gamma(0), \gamma(t))$, we now compute:

    \begin{align*}
        (\gamma(t),p)_o &= \frac{1}{2} \left( d(\gamma(t),o) + d(p, o) - d(\gamma(t),p)    \right) \\
        & \geq \frac{1}{2} \left( d(\gamma(t), \gamma(0)) -d (o, \gamma(0)) + d(p,\gamma(0)) - d(o, \gamma(0)) - d(\gamma(t),p)        \right) \\
        &= \frac{1}{2} \left( d(\gamma(t), \gamma(0)) -d (o, \gamma(0)) + d(p,\gamma(0)) - d(o, \gamma(0)) + d(\gamma(0), p) - d(\gamma(0),\gamma(t))     \right) \\
        &= d(\gamma(0),p) - d(\gamma(0),o) \\
        &\geq d(\gamma(0),\gamma(M)) - D - d(\gamma(0),o) \\
        &\geq \frac{1}{K}M-C - D - d(\gamma(0),o) \\
    \end{align*}

    Then $(\gamma(t),p)_o \geq \frac{1}{K}M-C - D - d(\gamma(0),o)$, and by replacing $p$ with $\gamma(M)$, we obtain $(\gamma(t),\gamma(M))_o \geq \frac{1}{K}M-C - 2D - d(\gamma(0),o)$.

    Let $r >0$ and choose the basepoint to be $o = \eta(0)$. Choose $M$ large enough so that $$\frac{1}{K}M - C -2D -B-2\delta > r.$$ By hypothesis, there exists some $N$ so that for all $n > N$, we have $d(\gamma_n(t), \eta(t)) < B$ for all $t \in [0,M]$. In particular, the preceding paragraph shows that the first and third terms of the minimum are larger than $r + 2\delta$. 
    
    We now bound the middle term of the above minimum.
    \begin{align*}
        (\gamma_n(M),\eta(M))_{\eta(0)} &= \frac{1}{2}\left( d(\gamma_n(M),\eta(0)) + d(\eta(M),\eta(0)) - d(\gamma_n(M),\eta(M))   \right) \\
        &\geq \frac{1}{2}\left( d(\gamma_n(M),\gamma_n(0)) - B + d(\eta(M),\eta(0)) - d(\gamma_n(M),\eta(M))   \right) \\
        &\geq \frac{1}{2} \left ( \frac{1}{K}M -C - B + \frac{1}{K}M - C - B \right) \\
        &= \frac{1}{K}M - C - B
    \end{align*}

    Then the second term of the minimum is also larger than $r + 2\delta$. 

    Then for all neighborhoods $U_r$ of $[\eta]$, there is some $N$ so that for all $n \geq N$, $[\gamma_n] \in U_r$, and hence $[\gamma_n] \to [\eta]$ in the topology of $\partial_\infty Z$.

\end{proof}

\subsection{$\partial \pi_Z$ factors through $\partial \pi_X$.}

We will now show that if $[Z] \in \mathcal{P}_+(G)$, then $\partial \pi_Z: \mathcal{V} \cong \partial_\infty X \times \partial_\infty Y \to \partial_\infty Z$ factors through the projection onto the first factor. The idea is that an element $t \in G$ with $h(t)>0$ acts on $\partial_\infty X $ and $\partial_\infty Z$ by expansions, while it acts on $\partial_\infty Y$ by contractions. Since the boundary map $\partial \pi_Z$ is equivariant, the differing dynamics is incompatible with sending elements of the same slice $\{\alpha\} \times \partial_\infty Y$ to different elements.

\begin{lemma}
    Let $[Z] \in \mathcal{P}_+(G)$. Then the map $\partial \pi_Z: \partial_\infty X \times \partial_\infty Y \to \partial_\infty Z$ factors through the projection onto the first factor.
\end{lemma}

\begin{proof}
    Let $t \in G$ satisfy $h(t) > 0$ and let $(\alpha_0, \beta_0) \in \partial_\infty X \times \partial_\infty Y \cong \mathcal{V}$ be the unique fixed point on the vertical boundary. Let $\beta \in \partial_\infty Y$ be any boundary point. Since $h(t) > 0$, $\alpha_0$ is repelling and $\beta_0$ is attracting. Then $t^k\cdot (\alpha_0, \beta) \to (\alpha_0,\beta_0)$ hence also $\partial \pi_Z (t^k\cdot (\alpha_0, \beta)) \to \partial \pi_Z(\alpha_0,\beta_0) =: \zeta_0$ by continuity. Equivariance gives that $t^k\partial \pi_Z(\alpha_0,\beta) \to \zeta_0$. Since $Z \in \mathcal{P}_+(G)$, we know that $\zeta_0 \in \partial_\infty Z$ is a repelling point of $t$ and therefore if $\partial \pi_Z (\alpha_0, \beta) \neq \zeta_0$, then $t^k \partial \pi_Z(\alpha_0,\beta)$ will not converge to $\zeta_0$. 

    The above paragraph implies that for every $\alpha \in \partial_\infty X$ which is an endpoint of some hyperbolic element of $G$, the map $\partial \pi_Z$ is constant on $\{\alpha\} \times \partial_\infty Y$. Since $G$ acts coboundedly on $X$, such endpoints are dense in $\partial_\infty X$, hence $\partial \pi_Z$ is constant on \textit{every} slice $\{\alpha\} \times \partial_\infty Y$, which is to say it factors through the projection onto the first factor. To see this last assertion, let $\alpha \in \partial_\infty X$ and choose boundary points $\alpha_n$ limiting to $\alpha$ such that $\alpha_n$ is an endpoint of some hyperbolic element of $G$. Then $\partial \pi_Z(\alpha_n,\beta) = \partial \pi_Z(\alpha_n,\beta')$ by the previous paragraph, and by taking the limit, $\partial \pi_Z(\alpha,\beta) = \partial \pi_Z(\alpha,\beta')$.
\end{proof}

\subsection{$X$ is maximal in $\mathcal{P}_+(G)$.}

We will now show that the action on $X$ is maximal in $\mathcal{P}_+(G)$. Let $[Z] \in \mathcal{P}_+(G)$ be a focal hyperbolic structure dominating $X$ via a coarsely Lipschitz, coarsely equivariant map $f:Z \to X$. We will explicitly construct a quasi-inverse to $f$.

Let $t \in h\inv([1,\infty))$ and let $W = \Axis_Y(t) \subset Y$. Then the subset $X \bowtie W \subset X\bowtie Y$ is coarsely isometric to $X$ (see \cite[Corollary 4.2.5.]{CaplingerLevitin}) with coarse isometry $i_W:X \to X\bowtie W$ given by $i_W(x) = (x, W(h(x)))$. Set $q = \pi_Z \circ i_W$. By coarse commutativity, $d_\infty(f\circ \pi_Z, \pi_X) < \infty$, where $\pi_X:X\bowtie Y \to X$ is the projection to $X$. Then $$f\circ q = f\circ \pi_Z \circ i_W \sim \pi_X \circ i_W =  \id_X,$$ which we think of as saying that $q$ is coarsely injective.

\begin{lemma}
\label{lem:q_coarsely_surjective}
    The above map $q$ is coarsely surjective.
\end{lemma}

\begin{proof}

Choose a basepoint $(x_0,y_0) \in \Axis_{X\bowtie Y}(t) \subset  X\bowtie Y$ and set $z_0 = \pi_Z(x_0,y_0)$. Let $z \in Z$. Then there is some $C>0$ not depending on $z$ and some $g\in G$ so that $d(g\cdot z_0,z) < C$. Let $D$ be the coarse equivariance constant and $M$ be the Morse constant for $(K,C)$-quasigeodesics, where $(K,C)$ is as in \cref{lem:vert_to_quasi}.

Then $z$ lies in the $C$ neighborhood of $g\cdot \pi_Z \Axis_{X\bowtie Y}(t)$, and therefore also in the $(C+D)$-neighborhood of $\pi_Z(g\cdot \Axis_{X\bowtie Y} (t)) = \pi_Z(\Axis_{X\bowtie Y} (gtg\inv))$ 

Let $(\alpha_0,\beta_0) \in \mathcal{V}$ be the fixed point of $t$ (the endpoints of $\Axis_{X\bowtie Y}(t)$). Then the endpoints of $\Axis(gtg\inv)$ are $(g\cdot \alpha_0,g\cdot \beta_0)$. Let $V_g\subset X$ denote the vertical line corresponding to $g\cdot \alpha_0$. Since $\partial \pi_Z$ factors through the projection onto the first factor, $\partial \pi_Z (g\cdot \alpha_0,g\cdot \beta_0) = \partial \pi_Z (g\cdot \alpha_0,\beta_0)$ and hence the quasigeodesics $\pi_Z(V_g \bowtie W)$ and $\pi_Z(\Axis_{X\bowtie Y}(gtg\inv))$ have the same endpoints in $\partial Z$ and therefore have Hausdorff distance at most $M$. Then $z$ is $(C+D+M)$ close to the image of $V_g$ under $q$, meaning $q$ is coarsely surjective. 

\end{proof}

That $q$ is the quasi-inverse of $f$ is now a simple exercise in coarse geometry.

\begin{lemma}
\label{lem:coarse_geo_exercise}
    If $f:Z \to X$ and $q:X\to Z$ are coarsely Lipschitz, coarsely surjective maps which satisfy $d_\infty(f\circ q, \id_X) < \infty $ then $f$ and $q$ are quasi-isometries.
\end{lemma}

We adopt the notation $x\sim y$ to mean that $d(x,y)$ is uniformly bounded in terms of constants implicit in the statement of the lemma.

\begin{proof}
    We first claim that $d_\infty(q\circ f, \id_Z) < \infty $. Let $z \in Z$. Then there is some $x \in X$ so that $q(x) \sim z$. Then $f(q(x)) \sim f(z)$, so $x\sim f(z)$. Applying $q$, we have $q(x) \sim q(f(z))$, so that $z \sim q(f(z))$. This proves the claim.

    We now show that $f$ is a quasi-isometry. The proof for $q$ is similar. Let $D$ the maximum of $d_\infty(q\circ f, \id_Z)$ and $d_\infty(f\circ q, \id_X)$ and let $(A,B)$ be the maximum of the coarse Lipschitz constants for $f$ and $q$. Let $z,w \in Z$. Since $q$ is coarsely Lipschitz, $$d_Z (q(f(z)), q(f(w))) \leq A d_X(f(z), f(w)) + B.$$ The left hand side differs from $d_Z(z,w)$ by at most $2D$, so we obtain the lower bound of $$\frac{1}{A}d_Z(z,w) - \frac{B}{A} -\frac{2D}{A} \leq  d_X (f(z),f(w)) \leq Ad_
    Z(z,w) + B. $$

\end{proof}

The coarse inverse of a coarsely equivariant quasi-isometry is also coarsely equivariant, (see \cite[Lemma 6.2.1.]{CaplingerLevitin}) so the two action of $G$ on $X$ and $Z$ are coarsely equivalent as coarse $G$-spaces. This shows that the poset $\mathcal{P}_+(G)$  has $X$ as a maximal element.

We are finally ready to prove Part 3 of \cref{thm:main}.

\begin{proof}[Proof of \cref{thm:main}]

    Lemmas 4.3, 4.4 and 4.5 of \cite{Valuations} together imply that the posets $\mathcal{P}_{\pm}(G)$ admit supremum operations. We are careful to note that these proofs, unlike their infimum counterparts, do not rely on the fact that the group $G = A\rtimes \Z$ in question has $A$ abelian. Combining \cref{lem:q_coarsely_surjective} and \cref{lem:coarse_geo_exercise}, we see that $X$ is a maximal element in $\mathcal{P}_+(G)$, so for any other element, $W \in \mathcal{P}_+(G)$, the supremum $\sup (X, W) = X$ dominates $W$. This proves that $[X] \in \mathcal{P}_+(G)$ is a maximum element. The proof that $[Y] \in \mathcal{P}_-(G)$ is maximal is identical.

\end{proof}

\section{$S$-arithmetic Sol: a group acting geometrically on a horocyclic product of mixed millefeuille spaces.}
\label{sec:New_Group}

In this section, we construct a group acting geometrically on a horocyclic product of two mixed millefeuille spaces, giving a partial answer to \cite[Question 1]{CaplingerLevitin}. This construction is an $S$-arithmetic analog of the classical example of $$\Z^2 \rtimes_B \Z \subset \R^2 \rtimes_{\Diag(e^t, e^{-t})} \R = \Sol = \h^2 \bowtie \h^2$$ for $|\tr(B)| > 2$. 

Let $p$ be prime and let $A \in \GL_2(\Z[1/p])$ be any matrix conjugate\footnote{One may take the diagonal matrix itself, but we take a general matrix to emphasize that the roles played by the real and $p$-adic parts are symmetric.} in $\GL_2(\Z[1/p])$ to $\begin{pmatrix}
    p & 0 \\ 0 & 1/p
\end{pmatrix}$. The group under consideration is $$G_p = \left(\Z\left[\frac{1}{p}\right]\right)^2 \rtimes_A \Z.$$ This group appears as $\Gamma_2/Z(\Gamma_2)$ in \cite{AbelsBrown87} and in Problem 1.12 of \cite{BCGS}. We now show that $G_p$ acts geometrically on a horocyclic product of mixed millefeuille spaces.


Let $\Q_p$ denote the $p$-adic numbers and let $\h_p^2$ denote the Heintze group $\R \rtimes_{p^t} \R$ with the Euclidean metric at the origin (this is a re-scaled hyperbolic plane with curvature $-(\ln{p})^2$). Let $F:\R\times \Q_p \to \R \times \Q_p$ be defined by $F(x,y) = (p\cdot x,y/p)$. Note that $F$ is expanding, and $(\R\times \Q_p) \rtimes_F \Z$ is an amenable, hyperbolic locally compact group. 

Identify $\Q_p$ with $\partial_\infty T_p$, the parabolic boundary of the regular $(p+1)$-valent tree, and note that every similarity of $\Q_p$ extends to an isometry on $T_p$. Then $(\R\times \Q_p) \rtimes_F\Z\subset \Aff(\R) \times \Aff(\Q_p)$ acts on both $\h^2_p$ and on $T_p$. These two actions have the same height function, so induce an isometric action on the mixed millefeuille space $\h^2_p[p]$. Since the action of $\R\times \Q_p$ is cobounded (in fact transitive) on the $0$-horosphere, the action of $(\R\times \Q_p) \rtimes_F\Z$ on $\h^2_p[p]$ is also cobounded. Let $\alpha: (\R \times \Q_p)^2  \to (\R \times \Q_p)^2 $ be defined by $\alpha(x,y) = (F(x), F\inv(y))$. Then $(\R \times \Q_p)^2 \rtimes_\alpha \Z $ acts coboundedly on $\h^2_p[p] \bowtie \h^2_p[p]$---the coboundedness can again be seen by noting that $(\R \times \Q_p)^2$ acts transitively on the $0$-horosphere. 

We now define an action of $G_p$ on $\h^2_p[p] \bowtie \h^2_p[p]$. Let $P_\R, P_{\Q_p} \in \GL_2(\Z[1/p])$ be matrices such that $$P_\R A P_\R\inv = \begin{pmatrix} p & 0 \\ 0& \frac{1}{p} 
\end{pmatrix} \quad \text{and} \quad P_{\Q_p}  A P_{\Q_p}\inv = \begin{pmatrix} \frac{1}{p} & 0 \\ 0 & p \end{pmatrix}.$$ After regrouping coordinates $(\R \times \Q_p)^2  = \R^2 \times \Q_p^2$, define $Q: \R^2 \times \Q_p^2 \to \R^2 \times \Q_p^2$ by $Q = \Diag(P_\R, P_{\Q_p})$, which is to say the linear map which acts by $P_\R$ on $\R^2$ and by $P_{\Q_p}$ on $\Q_p^2$. Similarly let $\overline{A} = \Diag(A,A)$. Then $\alpha = Q \overline{A} Q\inv $. Let $\iota: \Z[1/p]\to \R \times \Q_p$ denote the diagonal embedding. The identity $Q\overline{A} Q\inv = \alpha$ now implies that the map $f:G_p \to (\R \times \Q_p)^2 \rtimes_\alpha \Z$ defined by $(v,m) \to (Q(\iota(v)), m)$ is a homomorphism. 

Since $\iota(\Z[1/p]^2) \subset (\R\times \Q_p)^2$ is discrete, $$f(G_p) \subset (\R\times \Q_p)^2 \rtimes_\alpha \Z \subset \Isom(\h^2_p[p] \bowtie \h^2_p[p])$$ is also discrete. The action of $G_p$ on $\h^2_p[p] \bowtie \h^2_p[p]$ is therefore properly discontinuous. Let $C \subset \h^2_p[p] \bowtie \h^2_p[p]$ be a compact set whose $(\R\times \Q_p)^2 \rtimes_\alpha \Z$-translates cover $\h^2_p[p] \bowtie \h^2_p[p]$. Since $\iota(\Z[1/p]) \subset \R \times \Q_p$ is cocompact, there is also a compact $K \subset (\R \times \Q_p)^2 \rtimes_\alpha \Z$ whose $f(G_p)$-translates cover $(\R \times \Q_p)^2 \rtimes_\alpha \Z$. Then the $f(G_p)$-translates of $K\cdot C \subset \h^2_p[p] \bowtie \h^2_p[p]$ cover $\h^2_p[p] \bowtie \h^2_p[p]$, so the action of $G_p$ on $\h^2_p[p] \bowtie \h^2_p[p]$ is geometric.

\bibliographystyle{alpha}
\bibliography{bib}

\end{document}